\documentclass[11pt]{amsart}

\usepackage{a4wide, amsmath,  amssymb, amscd, mathptmx, amsthm}
\usepackage{amsfonts}
\usepackage{stmaryrd}
\usepackage{bbm}

\usepackage{color}

\usepackage{pstricks}
\usepackage[all]{xy}\CompileMatrices\SelectTips{cm}{12}

\theoremstyle{plain}
\newtheorem{Thm}{\sc Theorem}[section]
\newtheorem{theorem}[Thm]{\sc Theorem}
\newtheorem{corollary}[Thm]{\sc Corollary}
\newtheorem*{corollary*}{\sc Corollary}
\newtheorem{proposition}[Thm]{\sc Proposition}
\newtheorem*{proposition*}{\sc Proposition}
\newtheorem{lemma}[Thm]{\sc Lemma}

\theoremstyle{remark}
\newtheorem{remark}[Thm]{Remark}

\newtheorem*{example*}{Example}
\newtheorem*{remark*}{Remark}

\newcommand{\cE}{{\mathcal E}}

\newcommand{\cO}{{\mathcal O}}

\newcommand{\cS}{{\mathcal S}}

\newcommand{\cX}{{\mathcal X}}

\newcommand{\PP}{{\mathbb P}}

\newcommand{\RR}{{\mathbb R}}

\newcommand{\ZZ}{{\mathbb Z}}

\newcommand{\cEnd}{{\mathop{\mathcal{E}nd}\,}}

\newcommand{\Spec}{\mathop{\rm Spec \, }}

\newcommand{\Pic}{{\mathop{\rm Pic\, }}}
\newcommand{\rk}{\mathop{{\rm rk}}}

\begin{document}

\title{On boundedness of semistable sheaves}
\author{Adrian Langer}
\address{Institute of Mathematics, University of Warsaw,
	ul.\ Banacha 2, 02-097 Warszawa, Poland}
\email{alan@mimuw.edu.pl}


\thanks{
	The author was partially supported by Polish National Centre (NCN) contract number 
	2018/29/B/ST1/01232. }

\date{\today}

\maketitle
\medskip

\begin{abstract} 
	We give a new simple proof of boundedness of the family of semistable sheaves with fixed numerical invariants on a fixed smooth projective variety. In characteristic zero our method gives a quick proof of Bogomolov's inequality for semistable sheaves on a smooth projective variety of any dimension $\ge 2$ without using any restriction theorems. 
\end{abstract}

\section{Introduction}

Let $X$ be a smooth projective variety defined over an algebraically closed  field $k$
and let $H$ be an ample divisor on $X$. A family $\cS$ of isomorphism classes of coherent sheaves on $X$ is called bounded if there exists a scheme $S$ of finite type over $k$ and an $S$-flat family $\cE$ of coherent sheaves on the fibers of $X\times S\to S$, which contains isomorphism classes of all sheaves from $\cS$.  The following theorem is a crucial step in the construction of a projective moduli scheme of semistable sheaves with fixed numerical invariants:

\begin{theorem}\label{main}
	The family of isomorphism classes of slope $H$-semistable torsion free coherent sheaves on $X$ with fixed numerical invariants (e.g., with fixed rank and Chern classes) is bounded.
\end{theorem}

Many people contributed to the proof of this result including D. Gieseker, S. Kleiman, M. Maruyama, F. Takemoto and the author. 
The final version in characteristic zero was proven by M. Maruyama using the Grauert--M\"ulich type restriction theorem. 
In general, the above theorem was conjectured by M. Maruyama in late seventies and finally proven in \cite{La1}.  It allowed to finish the construction of the moduli scheme of semistable sheaves on $X$ in case $k$ has positive characteristic $p$. We refer to \cite{HL}, \cite{La1} and \cite{Ma} for more on the history of the problem and earlier results.

The proof of this result in \cite{La1} is rather complicated and it uses some purely characteristic $p$ methods through the study of Frobenius pullbacks of torsion free sheaves. Since its appearance there were no simplifications or  other proofs of this result. In this note we give a new simple proof of Theorem \ref{main} that avoids any characteristic $p$ methods. In fact, as in \cite{La1}, we prove a more general result that works in mixed characteristic (see Theorem \ref{boundedness}).

Our proof proceeds by reducing the statement to the case of projective spaces and proving Bogomolov's inequality by induction on the dimension using changes of polarization. The main novelty of this approach is that one can prove Bogomolov's inequality on projective spaces without using any restriction theorems. This in turn allows to prove restriction theorems and reduce the problem to lower dimensions. This method is new even  in the characteristic zero case. A small variant of this approach allows us to give a new simple proof of Bogomolov's inequality on higher dimensional varieties in characteristic zero without using any restriction theorems (see Theorem \ref{Bog-0}). This inequality implies effective restriction theorems for slope stability and slope semistability changing the usual logic and allowing to dispense with proving the Mehta--Ramanathan restriction theorems.

\section{Preliminaries}

In this section $X$ is a smooth projective $n$-dimensional variety defined over an algebraically closed field $k$. We assume that $n\ge 1$.

\subsection{Semistability}

Let $(L_1,..., L_{n-1})$ be a collection of nef divisors on $X$. Let us assume that the $1$-cycle $L_1...L_{n-1}$ is numerically nontrivial, i.e., there exists some divisor $D$ such that $DL_1...L_{n-1}\ne 0$. 
Let $E$ be  a rank $r$ torsion free sheaf on $X$.  Then we define the \emph{slope of $E$ with respect to} $(L_1,..., L_{n-1})$ as
$$\mu _{L_1...L_{n-1}} (E)= \frac{c_1(E) L_1...L_{n-1}}{r}. $$
We define $\mu _{\max, L_1...L_{n-1}} (E)$ as the maximum of $\mu _{L_1...L_{n-1}} (F)$ for all subsheaves $F\subset E$. Similarly, we define $\mu _{\min, L_1...L_{n-1}} (E)$ as the minimum of $\mu _{L_1...L_{n-1}} (F)$ for all torsion free quotients $F$  of $ E$. These are well defined rational numbers. 
We say that $E$ is \emph{slope  $(L_1,..., L_{n-1})$-semistable} if 
$\mu _{L_1...L_{n-1}} (F)\le \mu _{L_1...L_{n-1}} (E)$ for all subsheaves $F\subset E$ of rank less than $r$.
Similarly, we define \emph{slope  $(L_1,..., L_{n-1})$-stable} sheaves using the strict inequality
$\mu _{L_1...L_{n-1}} (F)<\mu _{L_1...L_{n-1}} (E)$.

In case $H$ is an ample divisor, we define $\mu_H(E)$, $\mu_{\max , H} (E)$, $\mu_{\min ,H}(E)$ and slope $H$-(semi)stability using the collection of $(n-1)$-divisors $(H, ..., H)$.  

\medskip

If $E$ is a rank $r$ torsion free sheaf on $X$ we define the \emph{discriminant} of $E$ as $\Delta (E)=2rc_2(E)-(r-1)c_1^2(E)$. We say that \emph{Bogomolov's inequality holds for the collection of nef divisors} $(L_1,...,L_{n-1})$ if for every  slope  $(L_1,..., L_{n-1})$-semistable sheaf $E$ we have $\Delta (E)L_2...L_{n-1}\ge 0$. Note that this notion depends on the order of nef divisors in the collection.

\subsection{The Hodge index theorem}

The Hodge index theorem implies that if $D_1$ and $D_2$ are divisors on a smooth projective surface 
and $(a_1D_1+a_2D_2)^2>0$ for some $a_1,a_2\in \RR$ then $D_1^2\cdot D_2^2\le (D_1D_2)^2$.
Now if $X$ is a smooth projective variety of dimension $n\ge 2$ we see that if $H_1,..., H_{n-2}$ are ample divisors and $D_1, D_2$ are such that  $(a_1D_1+a_2D_2)^2H_1...H_{n-2}>0$ for some $a_1,a_2\in \RR$ then $D_1^2H_1...H_{n-2}\cdot D_2^2H_1...H_{n-2}\le (D_1D_2H_1...H_{n-2})^2$. If $H_1,...,H_{n-2}$
are only nef we can replace $H_i$ by $H_i+tH$ for some ample $H$ and positive $t$. Passing with $t$ to $0$
we get the same result as above assuming that $H_1,...,H_{n-2}$ are only nef.
This allows us to prove the following version of the Hodge index theorem.

\begin{lemma}\label{HIT}
	Let $(L_1,...,L_{n-1})$ be a collection of nef divisors such that the $1$-cycle $L_1...L_{n-1}$ is numerically nontrivial. Then for any ample divisor $H$ we have $HL_1...L_{n-1}>0$. Moreover, if $DL_1...L_{n-1}=0$
	for some divisor $D$ then $D^2L_2...L_{n-1}\le 0$.	
\end{lemma}

\begin{proof}
	If $L_1...L_{n-1}$ is numerically nontrivial then there exists some divisor $M$ such that $ML_1...L_{n-1}> 0$.
	But then there exists some $m>0$ such that  $H^0(X, \cO_{X}(mH-M ) )\ne 0$. Therefore $mH L_1...L_{n-1}\ge ML_1...L_{n-1}> 0$, which proves the first claim. To prove the second claim let us take $D$ such that  $DL_1...L_{n-1}=0$.	The above version of the Hodge index theorem implies that for any $t>0$ and any ample $H$ we have
	$$D^2L_2...L_{n-1}\cdot (L_1+tH)^2L_2...L_{n-1} \le (D(L_1+tH)L_2...L_{n-1} )^2=t^2 (DHL_2...L_{n-1})^2. $$
	Passing with $t$ to $0$ we see that if $L_1^2L_2...L_{n-1}>0$ then $D^2L_2...L_{n-1}\le 0$.
	If $L_1^2L_2...L_{n-1}=0$ then dividing the above inequality by $t$ and passing with $t$ to $0$
	again gives $D^2L_2...L_{n-1}\le 0$.
\end{proof}

\section{Boundedness of semistable sheaves}

\subsection{Change of polarization}

The following proposition has a similar proof as \cite[Proposition 6.2]{La3}. Since it is crucial 
for the following arguments, we give all details of its proof for the convenience of the reader.

\begin{proposition}\label{pol-change}
	Let $X$ be a smooth projective variety of dimension $n$ and let $(L_1,...,L_{n-1})$ be a collection of nef divisors such that the $1$-cycle $L_1...L_{n-1}$ is numerically nontrivial. If Bogomolov's inequality holds for $(L_1,...,L_{n-1})$ then it holds for any  $(M,L_2,...,L_{n-1})$ such that $M$ is nef and  $ML_2...L_{n-1}$ is numerically nontrivial. 
\end{proposition}

\begin{proof}
	The proof is by induction on the rank of $E$ with rank $1$ being left to the reader. Let us assume that Bogomolov's inequality holds for  all sheaves of rank less than $r$ which are slope semistable with respect to some $ML_2...L_{n-1}$, where $M$ is  nef and  $ML_2...L_{n-1}$ is numerically nontrivial. Let us now fix some nef $M$ such that 
	$ML_2...L_{n-1}$ is numerically nontrivial and let $E$ be a slope $ML_2...L_{n-1}$-semistable torsion free sheaf of rank $r$. 	
	If $E$ is slope $L_1L_2...L_{n-1}$-semistable then 
	$\Delta (E)L_2...L_{n-1}\ge 0$ by the assumption that Bogomolov's inequality holds for $L_1...L_{n-1}$. So we can assume that  $E$ is not slope $L_1L_2...L_{n-1}$-semistable. In this case we consider $M_t=(1-t)M+tL_1$ for $t\in [0,1]$. 
	Let us note that $M_{t}L_2...L_{n-1}$ is numerically nontrivial, as otherwise for any ample $H$ we get $(1-t)HML_2...L_{n-1}+tHL_1...L_{n-1}=0$, which contradicts the first part of Lemma \ref{HIT}. Now we have the following lemma.
	
	\begin{lemma}\label{change-of-polarization}
		There exist some rational $t_0\in [0, 1)$ and slope $M_{t_0}L_2...L_{n-1}$-semistable torsion free sheaves $E'$ and $E''$ of ranks $r'$, $r''$ less than $r$, such that the sequence
		$$0\to E'\to E\to E''\to 0$$
		is exact and $\mu _{M_{t_0}L_2...L_{n-1}}(E')=\mu _{M_{t_0}L_2...L_{n-1}}(E'')=\mu _{M_{t_0}L_2...L_{n-1}}(E)$.
	\end{lemma}
	
	\begin{proof}
		Let $\cS$ be the set of all saturated subsheaves  $F\subset E$ of rank less than $r$ such that $\mu _{L_1L_2...L_{n-1}}(F)>\mu _{L_1L_2...L_{n-1}}(E)$. Note that  for any $F\in \cS$ we have $r! \mu _{L_1L_2...L_{n-1}}(F)\in \ZZ$ and $\mu _{L_1L_2...L_{n-1}}(F)\le \mu _{\max, L_1L_2...L_{n-1}}(E)$, so  the set $\{\mu _{L_1L_2...L_{n-1}}(F): F\in \cS\}$ 
		is finite. Let us take $E'\in \cS$ such  that the quotient
		$$s(F):= \frac{\mu _{ML_2...L_{n-1}}(E)-\mu _{ML_2...L_{n-1}}(F)} {\mu _{L_1L_2...L_{n-1}}(F)-\mu _{L_1L_2...L_{n-1}}(E)}$$
		attains the minimum among all $F\in \cS$. Such $E'$ exists since 
		$r!(\mu _{ML_2...L_{n-1}}(E)-\mu _{ML_2...L_{n-1}}(F))$ is a non-negative integer 
		(because by assumption $E$ is slope $ML_2...L_{n-1}$-semistable) and the denominator  takes only a finite number of positive values. 
		Let us set  $t_0=\frac{s(E')}{1+s(E')}$ so that $s(E')=t_0/(1-t_0)$.
		For any $F\subset E$  of rank less than $r$ we have 
		{ \small
			$$ \mu _{M_{t_0}L_2...L_{n-1}}(E)-\mu _{M_{t_0}L_2...L_{n-1}}(F)=(1-t_0) (\mu _{ML_2...L_{n-1}}(E)-\mu _{ML_2...L_{n-1}}(F))
			-t_0 (\mu _{L_1L_2...L_{n-1}}(F)-\mu _{L_1L_2...L_{n-1}}(E)).
			$$}
		This difference is clearly non-negative if $F\not \in \cS$ and $\ge 0$ if $F\in \cS$ with equality for $F=E'$.
		Therefore $E'$ and $E''=E/E'$ satisfy the required assertions.
	\end{proof}
	
	Now to finish proof of the proposition note that by the induction assumption  we have $\Delta (E') L_2...L_{n-1}\ge 0$ and $\Delta (E'') L_2...L_{n-1}\ge 0$.
	Therefore by the Hodge index theorem (see the second part of Lemma \ref{HIT}) we get
	$$\frac{\Delta (E) L_2...L_{n-1}}{r}=\frac{\Delta (E') L_2...L_{n-1}}{r'}+\frac{\Delta (E'') L_2...L_{n-1}}{r''}-\frac{r'r''}{r}\left(\frac{c_1(E')}{r'}-\frac{c_1(E'')}{r''}
	\right) ^2L_2...L_{n-1}\ge 0.$$
\end{proof}

\begin{remark}
	Lemma \ref{change-of-polarization} is rather standard (see, e.g., \cite[Lemma 4.C.5]{HL}) but we give all the details as the proof in \cite{HL} uses Grothendieck's lemma that fails in our case (it fails even in the surface case when both $L_1$ and $M$ are not ample).
\end{remark}

\subsection{Bogomolov's inequality and restriction theorem on projective spaces} 

\begin{theorem}\label{Bogomolov-proj-sp}
	Let $H$ be a hyperplane on $\PP ^n$ and let $E$ be a slope $H$-semistable torsion free coherent sheaf on $\PP ^n$. Then $\Delta (E)H^{n-2}\ge 0.$
\end{theorem}

\begin{proof}
	The proof is by induction on the dimension $n$ starting with $n=2$. In this case the proof follows by standard arguments involving the Riemann--Roch theorem. More precisely, we prove the inequality by induction on the rank $r$ of $E$. If $E$ is slope $H$-stable then $h^0(\cEnd E)=1$ and $h^2(\cEnd E)=h^0(\cEnd E(-3))=0$, so $\chi (\cEnd E)= -\Delta (E)+r^2\chi (\cO_{\PP ^2})\le 1$, i.e., $\Delta (E)\ge r^2-1\ge 0$.
	If $E$ is  slope $H$-semistable but not slope $H$-stable then there exists slope $H$-semistable torsion free
	sheaves  $E'$ and $E''$ of ranks $r'$, $r''$ less than $r$, such that the sequence
	$$0\to E'\to E\to E''\to 0$$
	is exact and  $\mu _{H}(E')=\mu _{H}(E'')=\mu _{H}(E)$.
	Then by the Hodge index theorem (or using the fact that $\Pic \PP^2$ is generated by $H$) and the induction assumption we get
$$\frac{\Delta (E)}{r}=\frac{\Delta (E')}{r'}+\frac{\Delta (E'') }{r''}-\frac{r'r''}{r}\left(\frac{c_1(E')}{r'}-\frac{c_1(E'')}{r''}
\right) ^2\ge 0,$$	
which finishes the proof.

	Now let us assume that $n\ge 3$.
	Let $E$ be a slope $H$-semistable torsion free coherent sheaf on $\PP ^n$. 
	Let $\Lambda \subset |\cO_{\PP ^n} (1)|$ be a general pencil of hyperplanes. Let $q: Y\to \PP^n$
	be the blow up of $\PP ^n$ in the base locus of $\Lambda$ and let $p: Y\to \Lambda=\PP^1$ be the canonical 
	projection.
	
	We claim that Bogomolov's inequality holds for the collection $(p^*\cO _{\Lambda }(1), q^*(H)^{n-2})$. 
	Namely, let $F$ be a torsion free sheaf on $Y$, which  is slope $p^*\cO _{\Lambda }(1)q^*(H^{n-2})$-semistable. 	
Existence of the flattening stratification (see \cite[Theorem 2.1.5]{HL}) implies that there exists a non-empty open subset $U\subset \PP^1$ such that $F$ is flat over $U$. Moreover, we can assume that for every $s\in U$ the restriction  $F_s$ to the fiber of $p$ over $s$ is torsion-free (cf.~\cite[Corollary 1.1.14 and Lemma 3.1.1]{HL}).  Since  $F$ is  slope $p^*\cO _{\Lambda }(1)q^*(H^{n-2})$-semistable, it is also slope $H^{n-2}_{\eta}$-semistable on the fiber $Y_{\eta}=\PP^{n-1}_{k(\eta)}$ of $p$ over the generic point $\eta\in \Lambda$ (if $F'\subset F_{Y_{\eta}}$
destabilizes $F_{Y_{\eta}}$ then we can extend it to a coherent subsheaf of $F$, which destabilizes $F$ with respect to $(p^*\cO _{\Lambda }(1), q^*(H)^{n-2})$). By openness of slope semistability for flat families
(cf. \cite[Propposition 2.3.1]{HL} and \cite[Chapter 1, Theorem 4.2]{Ma}), the set of $s\in U$ such that $F_s$ is slope semistable on the fiber $Y_s=\PP^{n-1} _{k(s)}$ is open. Since it contains the generic point $\eta$, it is non-empty and hence there exists a  closed (geometric) point $s$ of $U\subset \Lambda$ for which $F_s$ is slope semistable. Therefore by the induction assumption $\Delta (F)q^*(H^{n-2})=\Delta (F_s)H_s^{n-3}\ge 0$, where $H_s$ is a hyperplane in $Y_s$. This proves our claim.

 Now Proposition  \ref{pol-change} implies that Bogomolov's inequality holds also for $q^*(H)^{n-1}$. 
Since $E$ is torsion free, it is locally free outside of a closed subset of codimension $\ge 2$. So $E$ is locally free along the base locus of $\Lambda$ and thus $q^*E$ is torsion free. This implies that $q^*E$ is slope $q^*(H)^{n-1}$-semistable and hence $\Delta (E)H^{n-2} =\Delta (q^*E)(q^*H)^{n-2}\ge 0.$
\end{proof}

As in \cite[Theorem 5.1]{La1} the above theorem implies the following corollary:

\begin{corollary}\label{unstable-Bogomolov}
	Let $E$ be a torsion free rank $r$ sheaf on $\PP ^n$. 
	Then we have
	$$\Delta (E)H^{n-2}+ r^2(\mu_{\max , H}(E) -\mu _H (E))(\mu _H (E) -\mu_{\min , H} (E) )\ge 0.$$
\end{corollary}

\begin{proof}
	Let $0=F_0\subset F_1\subset \dots \subset F_m=E$ be the
	Harder--Narasimhan filtration of $E$ and let us set $F^i=F_i/F_{i-1}$, $r_i=\rk F^i$, $\mu _i =\mu _H(F^i)$.  Then by Theorem \ref{Bogomolov-proj-sp} we have
	$$\aligned
	\frac{\Delta (E)H^{n-2}}{ r}&= \sum \frac{\Delta
		(F^i)H^{n-2}}{ r_i} -\frac{1}{r}\sum_{i<j} r_ir_j \left(
	\frac{c_1F^i}{r_i}-\frac{c_1F^j}{ r_j}\right) ^2
	H^{n-2}\\
	&= \sum \frac{\Delta (F^i)H^{n-2}}{ r_i}-
	\frac{1}{ r } \sum_{i<j}r_ir_j(\mu_i-\mu_j)^2\ge -
	\frac{1}{ r } \sum_{i<j}r_ir_j(\mu_i-\mu_j)^2.\\
	\endaligned$$
	So the required inequality follows from \cite[Lemma 1.4]{La1}.
\end{proof}

As in \cite[Corollary 5.4]{La1} the above corollary implies the following restriction theorem (we state only a simplified version for slope semistability, but more precise versions need to be used in its proof).

\begin{theorem}\label{restriction}
	Let $E$ be a torsion free rank $r$ sheaf on $\PP ^n$, $n\ge 2$. Let $D\in |mH|$ be a general hypersurface of degree
	$$m>\frac {(r-1)^2\Delta (E) H^{n-2} +1}{r(r-1)} .$$
	If $E$ is slope  $H$-semistable then the restriction $E_D$ is slope $H_D$-semistable.
\end{theorem}

\begin{proof}
First, one proves a restriction theorem for slope stability following the proof of \cite[Theorem 5.2]{La1}
with $\beta_r=0$. Then one uses this theorem to factors of the Jordan--H\"older filtration of $E$ as in the proof of \cite[Corollary 5.4]{La1}. Finally, one uses the fact that a restriction of a torsion free sheaf to a general hypersurface of fixed degree is still torsion free (see \cite[Lemma 1.1.12 and Corollary 1.1.14]{HL}).
\end{proof}

\medskip

The above restriction theorem is sufficient for our proof of the boundedness result. However, we prefer to 
use similar methods to prove the following new restriction theorem that is slightly better suited for our proof
of Theorem \ref{boundedness}. 
To simplify exposition we omit polarization in the notation of slopes.

\begin{theorem}\label{restriction2}
	Let $E$ be a torsion free rank $r$ sheaf on $\PP ^n$, $n\ge 2$. Then for  a general hyperplane $D\in |H|$
	we have 
	$$\mu_{\max} (E_D)-\mu_{\max} (E)\le \frac{\Delta (E)H^{n-2}}{2r}+ \frac{r}{2}(\mu_{\max }(E) -\mu  (E))(\mu  (E) -\mu_{\min } (E) ).$$
\end{theorem}

\begin{proof} First let us consider the case when $E$ is slope $H$-stable.

\begin{lemma}\label{restriction3}
	If $E$ is slope $H$-stable and $r\ge 2$ then for any hyperplane $D\in |H|$
such that $E_D$ is torsion free,	we have 
	$$\mu_{\max} (E_D)-\mu (E)\le \frac{\Delta (E)H^{n-2}-1}{2r}.$$
\end{lemma}

\begin{proof}
Let $S\subset E_D$ be a saturated subsheaf of rank $\rho \ge 1$
and let $T=E_D/S$. Let  $E'$ be the kernel of the composition $E\to E_D\to T$.
Then we have
$$\aligned
\Delta (E')H^{n-2} &=\Delta (E)H^{n-2}-\rho
(r-\rho)  + 2r(r-\rho) (\mu (T)-\mu(E))\\
&=\Delta (E)H^{n-2}-\rho (r-\rho)  - 2r\rho (\mu (S)-\mu(E)).\\
\endaligned
$$
Since $E'\subset E$ and $E$ is
slope $H$-stable we have
$$
\mu _{\max} (E')-\mu (E')=\frac{r-\rho}{r}+\mu _{\max} (E')-\mu (E)
\le \frac{r-\rho}{r}
-\frac{1}{r(r-1)} .
$$ Similarly, since $E(-H)\subset E'$ we have
$$
\mu (E')-\mu _{\min} (E')=\frac{\rho}{r}+\mu (E(-H))- \mu _{\min} (E')
\le \frac{\rho}{ r}-\frac{1}{r(r-1)} .$$ 
Hence by Corollary
\ref{unstable-Bogomolov} we obtain
$$\aligned
0&\le \Delta (E')H^{n-2} + r^2(\mu_{\max }(E') -\mu  (E'))(\mu  (E') -\mu_{\min } (E') )\\
&\le  \Delta (E)H^{n-2}-\rho (r-\rho)- 2r\rho (\mu (S)-\mu(E)) +
\left(\rho  -\frac{1}{r-1}\right)
\left(r-\rho -\frac{1}{r-1}\right)\\
&= \Delta (E)H^{n-2}-\frac{r}{r-1} +\frac{1}{(r-1)^2}- 2r\rho (\mu (S)-\mu(E))\\
&\le \Delta (E)H^{n-2}-1 - 2r\rho (\mu (S)-\mu(E)).\\
\endaligned
$$
In particular, if $\mu (S)\ge \mu(E)$ then 
$$ \mu (S)-\mu(E) \le \frac{\Delta (E)H^{n-2}-1}{2r}.$$
If $E_D$ is not slope $H_D$-semistable then  taking for $S$ the maximal destabilizing subsheaf of $E_D$
 we get
$$\mu_{\max} (E_D)-\mu (E)\le \frac{\Delta (E)H^{n-2}-1}{2r}.$$
If $E_D$ is slope $H_D$-semistable but not slope $H_D$-stable then taking for $S$ any slope $H_D$-stable subsheaf of $E_D$ of the same slope, we get $\Delta (E)H^{n-2} \ge 1$, which gives the required inequality.
If $E_D$ is slope $H_D$-stable then $n>2$ (for $n=2$ the restriction $E_D$ is a direct sum of 
line bundles, so it is not stable). Then $\Delta (E)H^{n-2}=\Delta (E_D)H^{n-3}\ge 1$ follows from the previous cases by induction on the dimension $n$.
\end{proof}

Now we can deal with slope $H$-semistable sheaves.

\begin{lemma}\label{restriction4}
	If $E$ is slope $H$-semistable then for a general hyperplane $D\in |H|$ we have 
	$$\mu_{\max} (E_D)-\mu (E)\le \frac{\Delta (E)H^{n-2}}{2r}.$$
\end{lemma}

\begin{proof}
	Let $0=F_0\subset F_1\subset \dots \subset F_m=E$ be a Jordan--H\"older filtration of $E$ and let us set $F^i=F_i/F_{i-1}$ and $r_i=\rk F^i$.  Then for any $j$ Theorem \ref{Bogomolov-proj-sp} implies that
$$
\frac{\Delta (E)H^{n-2}}{ r}= \sum \frac{\Delta
	(F^i)H^{n-2}}{ r_i} -\frac{1}{r}\sum_{i<j} r_ir_j \left(
\frac{c_1F^i}{r_i}-\frac{c_1F^j}{ r_j}\right) ^2
H^{n-2}= \sum \frac{\Delta (F^i)H^{n-2}}{ r_i} \ge \frac{\Delta (F^j)H^{n-2}}{ r_j}.$$
Since by \cite[Lemma 1.1.12 and Corollary 1.1.14]{HL} a restriction of a torsion free sheaf to a general hyperplane $D\in |H|$ is still torsion free, Lemma \ref{restriction3} shows that 
	$$\mu_{\max} ((F^j)_D)-\mu (E)=\mu_{\max} ((F^j)_D)-\mu (F^j) \le \frac{\Delta (F_j)H^{n-2}}{2r_j}\le \frac{\Delta (E)H^{n-2}}{ 2r}.$$
Therefore we have
$$\mu_{\max} (E_D)\le \max _{j}\mu_{\max} ((F^j)_D)
\le  \mu (E)+ \frac{\Delta (E)H^{n-2}}{2r}.$$
\end{proof}

In general,  we consider  the Harder--Narasimhan filtration  $0=F_0\subset F_1\subset \dots \subset F_m=E$ of $E$. Let us set $F^i=F_i/F_{i-1}$ and $r_i=\rk F^i$. 
Then the proof of Corollary \ref{unstable-Bogomolov} shows that for any $j$ we have
$$\frac{\Delta (F_j)H^{n-2}}{r_j}\le \frac{\Delta (E)H^{n-2}}{ r}+ r(\mu_{\max }(E) -\mu (E))(\mu (E) -\mu_{\min} (E) ).$$
Since by  Lemma \ref{restriction3}
$$\mu_{\max} (E_D)\le \max _{j}\mu_{\max} ((F^j)_D)
\le \max _{j}\left( \mu (F^j)+ \frac{\Delta (F_j)H^{n-2}}{2r_j}\right) \le \mu _{\max} (E)+\max _{j}\frac{\Delta (F_j)H^{n-2}}{2r_j},$$
we get the required inequality.
\end{proof}

\begin{remark}
\begin{enumerate}
	\item 	Using \cite[Corollary 1.1.15]{HL} it is sufficient to assume in Theorems \ref{restriction}
	and \ref{restriction2} that $k$ is an infinite, not necessarily algebraically closed, field.
\item 
Note that a similar method as that of proof of Theorem \ref{restriction2} gives some inequalities on the slope of the maximal destabilizing subsheaf of the restriction $E_D$ for a general divisor $D\in |mH|$ for $m\ge 1$
on an arbitrary normal variety  (cf. \cite[Theorem 5.2, Corollary 5.4]{La1}). We leave the details of proof to the interested reader.
\end{enumerate}
\end{remark}

\subsection{Proof of boundedness of semistable sheaves}

Let $X _k$ be an $n$-dimensional projective scheme over 
an algebraically closed field $k$ and let $H=\cO _{X_k}(1)$ be an ample divisor on $X_k$.
If  $E$ is a coherent sheaf of pure dimension $d$ on $X_k$ then there exist integers $a_0(E),\dots ,a_d(E)$
and rational numbers $\alpha_0 (E), \dots ,\alpha_d(E)$  such that
$$\chi (X_k, E(m))=\sum _{i=0}^d a_i (E) {m+d-i \choose d-i}=\sum _{i=0}^d \alpha_i (E)\frac{m^i}{i!}.$$
One defines the \emph{generalized slope} of $E$ by $\hat\mu (E)=\frac{\alpha_{d-1}(E)}{\alpha_d (E)}=\frac{a_1(E)}{a_0(E)}+\frac{d+1}{2}$. It  is used
to define $\hat \mu _{\max} (E)$ in the same way as the usual slope is used to define $\mu_{\max}$, i.e., 
$\hat \mu _{\max} (E)$ is the maximum of $\hat \mu (F)$ for all subsheaves $F\subset E$.

\medskip

Let $f: X \to S$ be a projective morphism of noetherian schemes  of relative dimension $n$ and let $\cO _{X/S}(1)$ be an $f$-very ample line bundle on $X$. Let us consider the following families of sheaves.

\begin{enumerate}
	\item  Let $\cS _{X/S} (d; r,a_1,\dots ,a_d, \mu _{\max})$ be 
	the family of isomorphism classes of
	coherent sheaves on the fibres of $f$ such that $E$ on a geometric fibre $X_s$
	is a member of the family if $E$ is of pure dimension $d$,
	$\hat \mu _{\max} (E)\le \mu_{\max}$, $a_0(E)=r$, $a_1(E)=a_1$ and $a_i(E)\ge a_i$
	for $i\ge 2$.  
	\item  Let $\cS '_{X/S} (d; r,a_1,a_2, \mu _{\max})$ 
	be the family of isomorphism classes of
	coherent sheaves on the fibres of $f$ such that $E$ on a geometric fibre $X_s$
	is a member of the family if $E$ is of pure dimension $d$ and it satisfies Serre's condition $S_2$,
	$\hat\mu _{\max} (E)\le \mu_{\max}$, $a_0(E)=r$, $a_1(E)=a_1$ and $a_2(E)\ge a_2$.  
\end{enumerate}

The following theorem was first proven in \cite[Theorem 4.4]{La1}. Here we sketch a simple proof of this theorem based on  the results of the previous subsection.

\begin{theorem} \label{boundedness}
	The families $\cS _{X/S} (d; r,a_1,\dots ,a_d, \mu _{\max})$
	and $\cS '_{X/S} (d; r,a_1,a_2, \mu _{\max})$  are bounded. 
\end{theorem}

\begin{proof}
	For simplicity we consider only the family $\cS= \cS _{X/S} (d; r,a_1,\dots ,a_d, \mu _{\max})$ but a similar proof works also for the other family. We proceed by induction on $d$ with $d=0$ being trivial (for $d=0$,
$E\in \cS$ is $0$-dimensional with fixed $h^0(X, E)=r$; such $E$ are quotients of $\cO_X ^{\oplus r}$, so they form a bounded family). Let us assume that 
	the family  $\cS _{X/S} (d-1; r,a_1,\dots ,a_{d-1}, \mu _{\max})$ is bounded for all possible $X/S$ and all $r,a_1,\dots ,a_{d-1}, \mu _{\max}$.

	Theorem \ref{restriction2}  and our induction assumption imply that the family
	$\cS _{\PP ^d_S/S} (d; r,a_1,\dots ,a_d, \mu _{\max})$ is bounded for all $r,a_1,\dots ,a_{d}, \mu _{\max}$
	(this reduction step is a classical result of M. Maruyama; see \cite[Chapter 2, Proposition 3.4]{Ma}).  The remaining part of the proof depends on projection's method, which is originally due to C. Simpson and J. Le Potier (see proof of \cite[Chapter 2, Theorem 7.9]{Ma}). 
	Without loss of generality we can shrink $S$ so that $X/S$ embeds into $\PP ^N_S$ for some fixed $N$. 
	Let us fix a linear subspace  $T=\PP ^{N-d-1}_S\subset \PP ^N_S$.
	It is sufficient to bound the subfamily $\tilde \cS$ of $\cS$ that contains classes of all sheaves $E$ on the geometric fiber $X_s$ such that the support of $E$ does not intersect $T_s$ (as in the proof of  \cite[Chapter 2, Theorem 7.9]{Ma}, this follows from the fact that  $\mathrm {PGL} \,(N+1)$ acts on $\cS$ and every orbit meets the  subfamily $\tilde \cS$).  Let us fix such $E$.
	Taking the linear projection $\PP ^N_S\dasharrow \PP^d_S$ from $T$ and restricting it to the  scheme-theoretic support $Z$ of $E$,  we get a well defined finite morphism $\pi : Z\to \PP ^d_s$. Since $E$ is pure of dimension $d$, $\pi_*E$ is torsion free and we have
	$$H^i(X_s, E(m))= H^i(Z, E(m))=H^i (\PP ^d_s, \pi_*E\otimes \cO_{\PP^d_s}(m)).\leqno{(*)}$$
In particular, we get equality $\chi (X_s, E(m))=\chi (\PP ^d_s, \pi_*E\otimes \cO_{\PP^d_s}(m))$
of Hilbert polynomials.	Moreover, one can prove that
	$$\mu _{\max} (\pi_*E)- \mu  (\pi _*E)\le \hat \mu _{\max} (E)- \hat\mu  (E)+(a_0(E))^2$$
	(see \cite[Lemmas 6.2.1 and 6.2.2]{La4} or proof of \cite[Chapter 2, Lemma 7.11]{Ma} for a weaker, non-explicit estimate).  By definition, this implies that there exist some constants 
	$r',a_1',\dots ,a_d', \mu _{\max} '$ such that for all $E$ in $\tilde \cS$, the pushforward $\pi_*E$ is in the family $ \cS_{\PP}:=\cS _{\PP^d_S/S} (d; r',a_1',\dots ,a_d', \mu _{\max} ')$. We claim that boundedness of the family $ \cS _{\PP}$
	implies boundedness of $\tilde \cS$. This follows from the fact that boundedness of a family implies that there is a common bound on the Castelnuovo--Mumford regularity of sheaves in this family (see \cite[Lemma 1.7.6]{HL} or \cite[Chapter 1, Theorem 3.11]{Ma}). But equality $(*)$ implies then existence of a common bound on the  Castelnuovo--Mumford regularity of sheaves in the family $\tilde \cS$. So again using  \cite[Lemma 1.7.6]{HL} (or \cite[Chapter 1, Theorem 3.11]{Ma}), we get boundedness of the family $\cS$.
\end{proof}

The above result has many applications. Let us just recall that by the proof of \cite[Theorem 1]{Mo1} it implies the following  Bogomolov type inequality for strongly semistable sheaves.

\begin{corollary}
	Let $X$ be a smooth projective variety of dimension $n\ge 2$ defined over an algeraically closed field $k$ of characteristic $p>0$. Let $H$ be an ample divisor and let
	$E$ be a strongly $H$-semistable sheaf on $X$. Then we have $\Delta(E)H^{n-2}\ge 0$.
\end{corollary}

A generalization of the above result was first proven in \cite[Theorem 3.2]{La1} as part of the proof of Theorem \ref{boundedness}.

\subsection{Bogomolov's inequality in characteristic zero}

In characteristic zero a similar proof as that of Theorem \ref{Bogomolov-proj-sp} allows us to reduce Bogomolov's theorem in higher dimension to the surface case:

\begin{theorem}\label{Bog-0}
	Let $X$ be a smooth projective variety defined over an algebraically closed  field $k$ of characteristic $0$.
	Let $(D_1, ..., D_{n-1})$ be a collection of nef divisors on $X$ such that the $1$-cycle $D_1...D_{n-1}$ is numerically nontrivial. If $E$ is a slope $(D_1, ..., D_{n-1})$-semistable torsion free coherent sheaf on $X$
	then $$\Delta (E)D_2...D_{n-1}\ge 0.$$
\end{theorem}

\begin{proof}
	The proof is by induction on the dimension $n$ assuming the inequality in dimension $n=2$.
	
	Let $E$ be a slope  $(D_1, ..., D_{n-1})$-torsion free coherent sheaf on $X$ of dimension $n>2$.  By Proposition \ref{pol-change}  we can assume that $D_1$ is very ample. Let $\Lambda \subset |\cO_X(D_1)|$ be a general pencil of hyperplanes. Let $q: Y\to X$ be the blow up of $X$ in the base locus of $\Lambda$ and let $p: Y\to \Lambda=\PP^1$ be the canonical projection. As in the proof of Theorem \ref{Bogomolov-proj-sp}, the induction assumption implies that Bogomolov's inequality holds for $(p^*\cO _{\Lambda }(1),q^*D_2,...,q^*D_{n-1})$.  Therefore by Proposition  \ref{pol-change} it also holds for $(q^*D_1,... , q^*D_{n-1})$.
	But $q^*E$ is torsion free and it is slope $(q^*D_1, ..., q^*D_{n-1})$-semistable, so  
	$$\Delta (E)D_2...D_{n-1} =\Delta (q^*E)q^*(D_2...D_{n-1})\ge 0.$$
\end{proof}

\begin{remark}
	For $n>2$ Bogomolov's inequality is usually stated for collections of ample divisors and obtained by restricting to surfaces using  Mehta--Ramanathan's theorem (see \cite[Theorem 7.3.1]{HL}).
	This approach works well if $D_1,...,D_{n-1}$ are multiples of the same ample divisor $H$ giving $\Delta(E) H^{n-2}\ge 0$. However, to the author's knowledge there is no written account of Mehta--Ramanathan's restriction theorem for non-proportional ample divisors.  The only other approach to Theorem \ref{Bog-0} is that from \cite[Theorem 3.2]{La1}, where it is a part of a complicated induction procedure.
\end{remark}

\begin{remark}
	As in the case of Theorem \ref{restriction}, Theorem \ref{Bog-0} implies effective restriction theorems for slope stability and slope semistability  (see \cite[Theorem 5.2 and Corollary 5.4]{La1}).
	This approach makes proving the Mehta--Ramanathan restriction theorems \cite[Theorems 7.2.1 and 7.2.8]{HL} obsolete as we recover much stronger results.
\end{remark}

\section{Appendix: Bogomolov's inequality on surfaces}

In the proof of Theorem \ref{Bog-0} we used Bogomolov's inequality on smooth projective surfaces in the characteristic zero case. Contemporary accounts of proof of such an inequality usually assume the fact that symmetric powers of a semistable vector bundle are semistable (see, e.g., \cite[Theorem 3.4.1]{HL}).
Unfortunately, the proof of this fact on surfaces (usually) uses some kind of restriction theorem and reduction to the curve case (see, e.g.,  \cite[Corollary 3.2.10]{HL} for Maruyama's proof using the Grauert--M\"ulich theorem). Although the original approach of Bogomolov (see \cite[10.12, Corollary 2 and Lemma 7.2]{Bo}) does not use any restriction theorems, we give another simpler account of proof that is motivated by a quite similar proof from \cite{Gi}.

\subsection{Bounding the slope of the Frobenius pull-back}

Let $X$ be a smooth $n$-dimensional variety defined over an algebraically closed field of characteristic  $p>0$ and let $F_{X/k}:X'\to X$ be the geometric Frobenius morphism. Let $E$ be a quasi-coherent $\cO_{X'}$-module. Then $F_{X/k}^*E=F_{X/k}^{-1}E\otimes _{F_{X/k}^{-1}\cO_{X'}} \cO_{X}$ has a canonical connection defined by $\nabla_{can}(e\otimes f)=e\otimes df$.

\begin{theorem}{\rm (Cartier)}
	There is an equivalence of categories between the category of quasi-coherent $\cO_{X'}$-modules and the category
	of quasi-coherent $\cO_X$-modules with integrable $k$-connection and vanishing $p$-curvature.
	This equivalence is given by the functors sending an $\cO_{X'}$-module $E$ to $(F_{X/k}^*E, \nabla _{can})$ and 
	$(V, \nabla)$ to $V^{\nabla}=\{v\in V: \nabla(v)=0\}$.
\end{theorem}

The following proposition is a variant of some well-known results (see, e.g., \cite[Corollary 2.5]{La1}).

\begin{proposition}\label{bounding-Frobenius-pull-back}
	Assume $X$ is projective and  fix a very ample divisor $H$ such that $T_X(H)$ is globally generated.
	Then for any rank $r$ torsion free sheaf $E$ we have
	$$\mu _{\max, H}(F_X^*E)\le p \mu_{\max, H}(E)+(r-1)H^n.$$
\end{proposition}

\begin{proof}
	Let us first assume that $E$ is slope $H$-semistable and let $E_0=0\subset E_1\subset ...\subset E_m=F_X^*E$
	be the Harder--Narasimhan filtration of $F_X^*E$. Note that none of the subsheaves $E_i\subset F_X^*E$
	descends to a subsheaf of $E$, as it would contradict semistability of $E$. Therefore by Cartier's theorem
	these subsheaves are not preserved by the canonical connection $\nabla _{can}$ of $F_X^*E$. 
	It follows that the induced $\cO_X$-linear maps 
	$$E_i\to F_X^*E/E_i\otimes \Omega_X$$ 
	are nonzero. Since $T_X(H)$ is globally generated, we can embed $\Omega_X$ into $\cO_X(H)^{\oplus N}$ for some positive integer $N$.
	So for every $i=1,...,m-1$ we have non-zero maps
	$$E_i\to F_X^*E/E_i\otimes \cO_X(H)^{\oplus N},$$
	whose existence implies
	$$\mu (E_i/E_{i-1})\le \mu (E_{i+1}/E_i)+H^n.$$
	Summing these inequalities for $i=1,...,m-1$ we get 
	$$\mu_{\max}(F_X^*E)\le \mu_{\min}(F_X^*E)+(m-1)H^n\le \mu (F_X^*E)+(r-1)H^n=p \mu (E)+(r-1)H^n,$$
	which proves the required inequality.
	
	Now if $E$ is not $H$-semistable then we can apply the obtained
	inequality to all the quotients of the Harder--Narasimhan filtration
	of $E$. Since $\mu _{\max, H}(F_X^*E)$ is bounded from the above by
	the maximum of slopes of any filtration with semistable quotients, we
	immediately get the required inequality.
\end{proof}

\subsection{Bounding the number of sections}

The following proposition is a special case of a much more general and precise result (see \cite[Theorem 3.3]{La2}). However, the proof given below is completely elementary, whereas the one in \cite{La2} uses Bogomolov's inequality that we will prove using this proposition.

\begin{proposition} \label{bounding-number-of-sections}
	There exists an explicit function $f$ depending only on $r$ and $C$
	such that the following holds.  Let $X$ be a smooth projective
	surface and let $H$ be a very ample divisor on $X$.  If $E$ is a
	rank $r$ torsion free sheaf on $X$ with
	$$\mu_{\max, H}(E)\le C\cdot H^2$$
	then we have
	$$h^0(X, E)\le f(r, C) \cdot H^2.$$
\end{proposition}

\begin{proof} Since $h^0(X, E)\le h^0(X, E^{**})$ and $E^{**}$ is locally free, it is sufficient to prove the above proposition assuming that $E$ is locally free. The proof is by induction on the rank $r$. For $r=1$ 
	$E$ is a line bundle, so if $\mu_{\max, H}(E)= c_1(E)\cdot H \le C\cdot H^2$ then 
	$h^0(X, E(-(\lceil C\rceil +1)H))=0$. Therefore for a general divisor $Y\in |(\lceil C\rceil +1)H|$ we have
	$$h^0(X, E)\le h^0 (Y, E_Y)= (\lceil C\rceil +1)c_1(E)\cdot H\le C (\lceil C\rceil +1)\cdot H^2 ,$$
	so we can take $f(1,C)=  C (\lceil C\rceil +1)$.
	
	If $r=2$ then let us take any section $\cO_X\to E$  (if there are no sections there is nothing to prove) and its saturation $L\to E$. Let us set $L'=E/L$.
	By definition we have
	$$\mu (L) \le \mu _{\max}(E) \le C\cdot H^2,$$
	so 
	$$h^0(X, L)\le f(1,C)\cdot H^2.$$
	On the other hand, we have
	$$\mu (L')=2\mu (E)-\mu (L)\le 2 \mu (E)\le 2 \mu_{\max}(E)\le 2C\cdot H^2,$$ 
	so
	$$h^0(X, L')\le f(1,2C)\cdot H^2.$$
	This implies
	$$h^0(X, E)\le h^0(X, L)+h^0(X, L')\le  (f(1,C)+ f(1,2C))\cdot H^2.$$
	
	Now let us assume that the proposition holds for sheaves of rank less
	than $r$ (for some $r>2$) and let us consider a torsion free sheaf $E$ of rank $r$.
	
	If $h^0(X, E)\ne 0$ then let $E'$ be the image of the
	evaluation map $H^0(X, E)\otimes \cO_X\to E$. By definition $h^0(X,
	E')=h^0(X, E)$, so if $r'=\rk E'<r$ then we get $h^0(X, E)\le
	f(r',C)\cdot H^2$. 
	
	So we can assume that $\rk E'=r$. Replacing $E$ with $E'$ we can also
	assume that $E$ is globally generated.	Since $r>2$ the cokernel of a general section $\cO_X\to E$ is torsion free. Let us set $E'=E/\cO_X$. Let $F'\subset E'$ be the maximal destabilizing subsheaf of $E'$ and let $F\subset E$
	be the preimage of $F'$ in $E$. Let us set $r'=\rk F'$.
	Then 
	$$\mu (F)=\frac{r'}{r'+1}\mu (F')\le \mu_{\max} (E)\le C\cdot H^2,$$
	so 
	$$\mu_{\max}(E')\le \frac{r'+1}{r'}C\cdot H^2\le 2C\cdot H^2.$$
	By the induction assumption we have
	$$h^0(X, E')\le f(r-1, 2C)\cdot H^2,$$
	so $h^0(X, E)\le (1+f(r-1, 2C))\cdot H^2$.
	
	Summing up, we see that as $f(r,C)$ we can take  $f(1, C)+f(1, 2C)$ for $r=2$ and $1+f(r-1, 2C)$ for $r>2.$
\end{proof}

\subsection{Bogomolov's inequality on surfaces}

\begin{theorem}{\rm (Bogomolov, \cite[10.12, Corollary 2 and Lemma 7.2]{Bo})} \label{Bogomolov-for-surfaces}
Let $X$ be a smooth projective surface defined over an algebraically closed field $k$ of characteristic $0$.
Let $H$ be a very ample divisor and let $E$ be a torsion free sheaf on $X$.
If $E$ is slope $H$-semistable then $\Delta (E)\ge 0$. 
\end{theorem} 

\begin{proof}  
	Since $\Delta (E)\ge \Delta (E^{**})$, replacing $E$ with $E^{**}$ we
	can assume that $E$ is locally free.  
	
	Let us first assume that $\det E\simeq \cO_X$. Without loss of
	generality we can assume that $T_X(H)$ is globally generated (we can
	always replace $H$ by its large multiple). There exists  a finitely
	generated over $\ZZ$ subring $R\subset k$, a smooth projective scheme $\cX\to S=\Spec R$ and a locally free sheaf $\cE$ on $\cX$ such that $\cX\otimes _R k=X$ and $\cE\otimes _R k=E$.	
	By Propositions 	\ref{bounding-Frobenius-pull-back} and
	\ref{bounding-number-of-sections} there exists a constant
	$C$ such that for all closed points $s\in S$ we have
	$$h^0(\cX_s, F_{\cX_s}^*\cE_s)\le C$$
	and 
	$$h^2(\cX_s, F_{\cX_s}^*\cE_s)=h^0(\cX_s, F_{\cX_s}^*\cE^*_s\otimes \omega_{\cX_s})\le C.$$
	Therefore 
	$$\chi (\cX_s,F_{\cX_s}^*\cE_s)\le 2C.$$
	But by the Riemann--Roch theorem we have
	$$\chi (\cX_s,F_{\cX_s}^*\cE_s)=r\chi (\cX_s, \cO_{\cX_s})-\int _{\cX_s}c_2(F_{\cX_s}^*\cE_s)= r\chi (X, \cO_{X})- p_s^2\cdot \int _Xc_2(E),$$
	where $p_s$ is the characteristic of the residue field
	$k(s)$. Therefore taking $p_s\gg 0$ we get $\int _X c_2(E)\ge 0$, as required.
	
	\medskip
	
	To prove the inequality in the general case we need the following covering lemma due to Bloch and Gieseker
	(see \cite[Theorem 3.2.9]{HL}).
	
	\begin{lemma}
		Let $X$ be a smooth projective variety defined over some algebraically closed field $k$ and $L$ be a line bundle on $X$. Let us fix a positive integer $r$. Then there exists a smooth projective variety $Y$, a finite surjective morphism
		$f: Y\to X$ and a line bundle $M$ on $Y$ such that $M^{\otimes r}\simeq f^*L$.
	\end{lemma}
	
	Let us apply this lemma to $L=\det E$. Then $E'=f^*E\otimes M^{-1}$ is a slope $f^*H$-semistable vector bundle with $\det E'\simeq \cO_Y$ (see, e.g., \cite[Lemma 3.2.2]{HL}). Therefore $\Delta (E')\ge 0$ by the first part of the proof. But $\Delta (E)=\Delta (E')/\deg f$, so $\Delta (E)\ge 0$.
\end{proof}

\medskip

\section*{Acknowledgements}
The author would like to thank Burt Totaro for some corrections to the first version of the paper.
The author would also like to thank the referee for forcing him to make the paper more readable.

The proofs in the appendix come from the author's lecture notes for various summer schools in 2016: Higher Dimensional Algebraic Geometry, Salt Lake City, USA,  Spring School on Moduli Theory, Shanghai, China and 
School on Higgs bundles and Fundamental Groups of Algebraic Varieties, Essen, Germany.

\end{document}